\documentclass[11pt]{amsart}
\usepackage[utf8]{inputenc}
\usepackage{amsmath}
\usepackage{amsthm}
\usepackage{amssymb}
\usepackage[T1]{fontenc}

\usepackage{fancyhdr}
\usepackage{hyperref}
\usepackage[x11names]{xcolor}
\usepackage{tikz}
\usetikzlibrary{fadings}

\usepackage{microtype}

\newcommand {\zmod}[1]{\mathbb{Z}/_{\!#1}}

\newtheorem{theorem}{Theorem}
\newtheorem*{theorem*}{Theorem}
\newtheorem{corollary}[theorem]{Corollary}
\newtheorem{proposition}[theorem]{Proposition}
\newtheorem{lemma}[theorem]{Lemma}

\theoremstyle{definition}

\newtheorem*{remark}{Remark}

\title[Exotic actions on product manifolds]{Constructions of exotic actions on product manifolds with an asymmetric factor}
\author{Zbigniew Błaszczyk and Marek Kaluba}

\thanks{The authors have been supported by the National Science Centre grants 2014/12/S/ST1/00368 and 2015/19/B/ST1/01458.}
\date{}

\begin{document}
\maketitle
\begin{abstract}
We explore transformation groups of manifolds of the form \mbox{$M\times S^n$}, where $M$ is an asymmetric manifold, i.e. a manifold which does not admit any non-trivial action of a finite group. In particular, we prove that for $n=2$ there exists an infinite family of distinct non-diagonal effective circle actions on such products. A similar result holds for actions of cyclic groups of prime order. We also discuss free circle actions on $M \times S^1$, where $M$ belongs to the class of ``almost asymmetric'' manifolds considered previously by V.~Puppe and M.~Kreck.
\end{abstract}

\section{Introduction}

A ubiquitous problem in transformation groups is to decide whether certain groups can act in a specific manner on a given class of spaces. We investigate an instance of this situation in the context of diagonal actions on manifolds given as Cartesian products.

Let $X$ and $Y$ be smooth $G$-manifolds. \textit{Diagonal action} $\Delta$ on the product $X\times Y$ is given by the formula
\[ \Delta\big(g,(x,y)\big) = (gx, gy) \textnormal{ for $g \in G$, $x \in X$ and $y \in Y$.} \]
More generally, an action $\phi\colon G\times (X \times Y)\to X\times Y$ is said to be a \emph{product action} provided that $\phi$ is smoothly conjugate to a diagonal action, i.e. there exists a $G$\nobreakdash-\hspace{0pt}equivariant diffeomorphism $f\colon X\times Y\to X\times Y$ satisfying $f\circ\phi = \Delta \circ (\textnormal{id}_G\times f)$.

The starting point for our work is the following question: \textit{How to distinguish whether an arbitrarily given $G$-action on $X \times Y$ is a product one?} Intuitively, products of manifolds with high degree of symmetry should admit plenty of exotic\footnote{We follow Adem and Smith \cite{AdemSmith2001} and call an action ``exotic'' when it is not conjugate to a product action.} actions. Indeed, work on the rank conjecture for products of spheres shows that actions on products can be very complicated, see e.g. Hambleton and \"{U}nl\"{u} \cite{HambletonUnlu}. However, as degree of symmetry of one of the manifolds in the product decreases, there might be a ``threshold dimension'' below which all actions become conjugate to product ones. As spheres have the highest possible degree of symmetry (Eisenhart \cite{Eisenhart}), the following seems to be a very natural testing example.
\begin{quote}
\emph{Let $M$ be an asymmetric manifold, i.e. a manifold which does not admit any non-trivial action of a finite group. What are the possible actions of $G$ on $M\times S^n$, depending on $G$ and $n$?}
\end{quote}

There is a strong sentiment expressed, among others, by Raymond and Schultz \cite{Browder1978}, Kreck \cite{Kreck2009} and Puppe \cite{Puppe2007} that ``most'' manifolds are asymmetric. On the other hand, in private communication Pawałowski challenges (teasingly) this view by suggesting that there should be at least as many symmetric manifolds as asymmetric ones. Indeed, a product of an asymmetric manifold with one that has symmetries yields a manifold that clearly again has symmetries.

The purpose of this note is to shed some light on the influence of the highly symmetric factor on symmetries of the whole product. In particular, we prove that for every asymmetric manifold $M$ there exist smooth actions of the circle group (or a cyclic group of prime order) on $M\times S^2$ which do not come solely from actions on~$S^2$, and a similar result follows for involutions on $M\times S^1$ (Theorems \ref{thm:non-product-actions} and {\ref{thm:Z2-non-product-actions}, respectively). This seems to stand in contrast with the case of free actions of the circle group on a class of $7$-manifolds given as $M\times S^1$, where $M$ is an ``almost asymmetric'' manifold: we present some evidence that such actions are product ones (Proposition~\ref{prop:product-actions}). We also discuss a possible approach to the problem of detecting product actions among free orientation preserving actions of cyclic groups of prime order on $M \times S^1$ (Theorem~\ref{thm:detecting_non-product_actions}).\medskip

\noindent\textbf{Notation.} We work in the smooth category. In particular, all manifolds and actions are taken to be smooth. Furthermore, unless otherwise stated, all manifolds are assumed to be compact.

\section{Asymmetric and ``almost asymmetric'' manifolds}
\label{sec:Asymmmetric_manifolds}

A closed manifold is said to be \textit{asymmetric} if it does not admit any non-tri\-vial action of a finite group. Such manifolds appear as soon as in dimension~$3$ in the shape of bundles over the circle with fibre an orientable surface of genus at least~$3$ (Raymond and Tollefson~\cite{Raymond1976}). Every $3$\nobreakdash-\hspace{0pt}dimensional asymmetric manifold is also aspherical, i.e. it has a contractible universal cover. Borel~\cite{Borel1983} provided a sufficient criterion for an aspherical manifold $M$ to be asymmetric: its $\pi_1$ is required to have trivial center and torsion free outer automorphism group $\textrm{Out}\big(\pi_1(M)\big)$.

Non-aspherical asymmetric manifolds can be observed in dimension~$4$ (Bloomberg~\cite{Bloomberg1975}). Curiously enough, all of known examples have non-trivial fundamental groups: the existence of a simply connected asymmetric manifold is an open problem.

Let us now briefly describe the ``least symmetric'' simply connected manifolds known at present. Consider a map $f \colon \mathbb{Z}^6 \to \mathbb{Z}$ (attributed to A.~Iarrobino by Puppe \cite{Puppe1995}) given by the polynomial
\begin{align*}
f(x_1,\ldots,x_6)  =& \;6\big(x_1x_4^2 - x_1^2x_4 +x_2x_4^2 + x_2x_4^2 -
 x_2^2x_5 + x_2x_5^2 \\
& + x_3^2x_4 - x_3x_4^2 + x_3^2x_6 + x_3x_6^2 + x_5^2x_6 + x_5x_6^2\\
& + x_1x_2x_4 +x_1x_2x_5 + x_1x_3x_6 +x_2x_4x_6 + x_3x_5x_6\\
&+ x_4x_5x_6 +
x_4x_5x_6 + x_4^3 +x_6^3\big).
\end{align*}

\noindent This map determines a trilinear symmetric form, hence also a cup product-like structure
$\mu\colon\mathbb{Z}^6 \oplus \mathbb{Z}^6 \oplus \mathbb{Z}^6 \to \mathbb{Z}$. It follows from Wall~\cite{Wall1966} that there exists an infinite family $\mathcal{M}_\mu$ consisting of simply connected $6$-dimensional closed manifolds $M$ such that $H^2(M;\mathbb{Z})\cong \mathbb{Z}^6$, $H^3(M;\mathbb{Z}) = 0$, and the iterated product $H^2(M;\mathbb{Z})\times H^2(M;\mathbb{Z})\times H^2(M;\mathbb{Z}) \to \mathbb{Z}$ is prescribed by $\mu$. All such $M$'s are distinguished by the first Pontryagin class. An appropriate choice of this class allows to exclude the existence of non-trivial actions of finite groups on $M$:

\begin{theorem*}[Puppe \cite{Puppe1995}]
There exists an infinite family $\mathcal{M}_{As}\subset \mathcal{M}_\mu$ of manifolds with the property that any $M \in \mathcal{M}_{As}$ does not admit any non-trivial action of a finite group, with the possible exception of an orientation reversing involution.
\end{theorem*}

In what follows we will frequently refer to manifolds from the family $\mathcal{M}_{As}$ as \textit{almost asymmetric}. Olbermann \cite{Olbermann2011} showed that these manifolds in fact admit smooth involutions with $3$\nobreakdash-dimensional fixed point sets.

\begin{remark}
We reiterate that we are discussing smooth manifolds. Certain non-smoothable simply connected $6$\nobreakdash-\hspace{0pt}manifolds were shown to be asymmetric by Kreck \cite{Kreck2009, Kreck2011}, albeit in a different sense: they admit no actions with equivariant tubular neighbourhoods.
\end{remark}

\section{Constructing exotic actions on $M \times S^2$}
\label{sec:Exotic-actions}

Throughout this section $G=S^1$ or $G=\mathbb{Z}_{p}$, $p$ a prime.

\begin{theorem} \label{thm:non-product-actions}
If $M$ is an $m$-dimensional asymmetric manifold for $m\geqslant 3$, then there exists an effective exotic $G$\nobreakdash-action on $M\times S^2$.
\end{theorem}

We provide a proof for $G=\mathbb{Z}_{p}$. The case of $G=S^1$ is analogous. Our approach was inspired by Hsiang \cite{Hsiang1964}.

\begin{proof}
Recall that any $G$\nobreakdash-action on $S^2$ is linear, hence any non-trivial product $G$-action on $M\times S^2$ has the fixed point set either diffeomorphic to \mbox{$M\sqcup M$} or empty (the latter can happen only if $p=2$).
Thus in order to prove the theorem, it suffices to construct a $G$-action on $M\times S^2$ with a non-empty fixed point set different from $M\sqcup M$.

Consider a non-simply connected $m$-dimensional homology sphere $\Sigma$ that bounds a contractible manifold $X$ (see Kervaire \cite[Theorem 3]{Kervaire1969}).
Take any non-trivial orthogonal $2$\nobreakdash-dimensional $G$\nobreakdash-representation $V$ and equip the product $X\times D(V)$ with a diagonal $G$-action by setting
\[g\!\cdot\! (x,y)\mapsto (x,gy) \textnormal{ for any $g\in G$, $x \in X$ and $y \in D(V)$.} \]
Since $m+3\geqslant 6$ and $\partial \big(X\times D(V)\big)$ is simply connected by the Seifert--van Kampen theorem, the $h$-cobordism theorem yields a diffeomorphism
\[ X \times D(V)\cong D^{m+3}.\]
The action restricted to the boundary sphere $S^{m+2}$ has the following two properties:
\begin{enumerate}
\item[(1)] $(S^{m+2})^G=\Sigma$, and
\item[(2)] the tangential $G$-module at any point $x \in \Sigma\subset S^{m+2}$ is isomorphic to $\mathbb{R}^m\oplus V$.
\end{enumerate}
To conclude the proof, identify $M\times S^2 \cong M\times S(V\oplus \mathbb{R})$, choose a fixed point $y\in M\times \big\{(0,0,1)\big\}$ and consider the equivariant connected sum along disk $G$-neighborhoods of $x$ and $y$:
\[{S^{m+2}}\#\big(M\times\! S(V\oplus \mathbb{R})\big)\cong M\times S^2. \]
Clearly, the induced action on the right hand side realizes $M\sqcup (M\#\Sigma)$ as the fixed point set. But $M$ and $M\#\Sigma$ have non-isomorphic fundamental groups, e.g. by the Gru\-shko\nobreakdash--von Neu\-mann theorem.
\end{proof}

\begin{theorem}\label{thm:Z2-non-product-actions}
If $M$ is an $m$-dimensional asymmetric manifold for $m\geqslant 4$, then there exists an exotic involution on $M\times S^1$.
\end{theorem}

\begin{proof}
When $G = \mathbb{Z}_{2}$, one can take $V$ to be the sign representation, and the proof then follows exactly as above, although $M$ has to be at least $4$\nobreakdash-\hspace{0pt}dimensional in order to comply with the assumptions of the $h$-cobordism theorem.
\end{proof}

\begin{remark} We emphasize that the construction above produces non\nobreakdash-\hspace{0pt}conjugate actions for different choices of $\Sigma$. For a plausible explanation of a connection between representations and the existence of exotic actions, see the discussion at the end of the paper.
\end{remark}

The following general strategy for detecting exotic actions can be extracted from the proof of Theorem \ref{thm:non-product-actions}.

\begin{proposition}\label{prop:Detection-non-product}
Let $M$ be a manifold such that the fixed point set $M^G$ is connected for any $G$\nobreakdash-action and let $H$ be a non-trivial closed subgroup of $G$. If $G$ acts on $M\times S^n$, $n \geq 1$, with the $H$-fixed point set
\[\big(M\times S^n\big)^H \cong X\sqcup Y\]
for $X$ not homotopy equivalent to $Y$, then the $G$\nobreakdash-action on $M\times S^n$ is exotic.
\end{proposition}

\begin{proof}
Since all subgroups $H \subseteq S^1$ are normal, the argument for $G = S^1$ may be restricted to actions of $S^1\!/H\cong S^1$, and those clearly have non-empty fixed point sets.

By Smith Theory, the fixed point set of any $G$\nobreakdash-action on $S^n$ is either a $\textnormal{mod}\,p$ homology sphere (when $G=\mathbb{Z}_{p}$), or an integral homology sphere (when $G=S^1$), say $\Sigma$. Therefore the fixed point set of a product $G$-action on $M \times S^n$ is equal to $M^G \times\Sigma$. This space is either connected or homotopy equivalent to $M^G \sqcup M^G$, and the conclusion follows.
\end{proof}

\begin{remark}
The results of this section carry over to topological and piecewise\hspace{0pt}-\hspace{0pt}linear categories.
\end{remark}

\section{Free actions on $M \times S^1$}

Throughout this section $M \in \mathcal{M}_{As}$.

\subsection{Free circle actions on $M\times S^1$}\label{sec:Actions_on_MxS^1}

\begin{proposition}\label{prop:product-actions}
Every free $S^1$\nobreakdash-action on $M\times S^1$ such that $\pi_1\big((M\times S^1)/S^1\big)$ is not the infinite cyclic group is a product action.
\end{proposition}

Note that the product $S^1$\nobreakdash-action on $M\times S^1$ in the statement of the above theorem is an action conjugate to the one given by
\[g(x,z)\mapsto (x,g\cdot z) \textnormal{ for any $g$, $z \in S^1$ and $x\in M$}, \]
where $\cdot$ is the complex multiplication.

\begin{proof}
A free $S^1$-action on $M \times S^1$ yields a fiber bundle 
\[\xi = \big(S^1 \stackrel{j_{\xi}}{\longrightarrow} M\times S^1 \stackrel{\pi_{\xi}}{\longrightarrow} X\big) \]
over the orbit space $X = (M \times S^1)/S^1$. The isomorphism class of $\xi$ is determined by the classifying map $f_{\xi} \colon X \to BS^1$. In particular, if $C$ denotes the preferred generator of $H^2(BS^1;\mathbb{Z})$, then $f_{\xi}$ is homotopic to a constant map if and only if the first Chern class $c_1(\xi) = f_{\xi}^*C$ is trivial. 

Temporarily assume that $c_1(\xi) =0$. Then $\xi$ is isomorphic to the product bundle $\varepsilon = (S^1 \stackrel{j_{\varepsilon}}{\longrightarrow} X \times S^1 \stackrel{\pi_{\varepsilon}}{\longrightarrow} X)$, and there is a commutative diagram
\begin{equation*}\label{eqn:diagram}
\begin{tikzpicture}[scale=1]
\draw (0,2) node(MxS1) {$M\times S^1$};
\draw (0,0) node(MxS12) {$X\times S^1$};
\draw (-3,1) node(S1) {$S^1$};
\draw (3,1) node(X) {$X$,};

\path [->] (S1) edge node [above] {$j_{\xi}$} (MxS1);
\path [->] (S1) edge node [below] {$j_{\varepsilon}$} (MxS12);

\path [->] (MxS1) edge node [above] {$\pi_{\xi}$} (X);
\path [->] (MxS12) edge node [below] {$\pi_{\varepsilon}$} (X);

\path [->] (MxS1) edge node [left]{$\psi$} (MxS12);

\end{tikzpicture}
\end{equation*}
with $\psi \colon M \times S^1 \to X \times S^1$ a diffeomorphism. Since $\pi_1(M) \times \pi_1(S^1) \cong \pi_1(X) \times \pi_1(S^1)$ and $\pi_1(M)=0$, it follows that $X$ is  simply connected. In addition, the composition $\pi_{\varepsilon} \circ \psi \circ i \colon M \to X$, where $i \colon M \to M\times S^1$ is the canonical inclusion, induces an isomorphism $\pi_k(M) \to \pi_k(X)$ for any $k\geq 2$, and therefore the map $M \to X$ is a homotopy equivalence. This reduces the proof of Proposition \ref{prop:product-actions} to the following two claims.\medskip

\begin{description}
 \item [Claim 1] \textit{The first Chern class $c_1(\xi)$ is trivial.}
 \item [Claim 2] \textit{The equivalence $M\to X$ can be improved to a diffeomorphism.}
\end{description}\medskip

\noindent\emph{Proof of Claim 1.}
Recall that $M$ is a simply connected $6$\nobreakdash-dimensional manifold satisfying $H^2(M)\cong \mathbb{Z}^6$ and $H^3(M) = 0$. (For the reminder of the proof we suppress integer coefficients from notation.) Inspection of the long exact sequence of homotopy groups of $\xi$ shows that $\pi_1(X) \cong \mathbb{Z}_m$, $1 \leq m < \infty$, as the infinite cyclic possibility is excluded by assumption. Furthermore, $\xi$ is orientable: indeed, since $\pi_1(BS^1) = 0$, the classifying bundle $ES^1 \to BS^1$ is orientable, and any pullback of an orientable bundle is again orientable. Hence we can consider the Gysin sequence
\[ \cdots \to H^i(X) \stackrel{\rho_i}{\longrightarrow} H^{i+2}(X) \to H^{i+2}(M\times S^1) \to H^{i+1}(X) \to \cdots, \]
where each $\rho_i$ is given by multiplication with the Euler class $e \in H^2(X)$, which in this case coincides with the Chern class $c_1(\xi)$.
From the part 
\[ 0 = H^7(X) \to H^7(M \times S^1) \to H^6(X) \to H^8(X) = 0 \]
we infer that $H^6(X)\cong H^7(M \times S^1) \cong \mathbb{Z}$, and therefore $X$ is orientable. By Poincar\'e duality, $H^5(X) \cong H_1(X) \cong \mathbb{Z}_m$. The part
\[ 0 = H^1(X) \to H^1(M\times S^1) \to H^0(X) \to \textrm{Im}\,\rho_0 \to 0\]
shows that $e$ is a torsion element and $\textrm{Im}\,\rho_i$ is a torsion group for any $i \geq 0$. The part
\[ 0 \to \textrm{Im}\,\rho_0 \to H^2(X) \to H^2(M \times S^1) \to H^1(X) = 0 \]
shows that 
\begin{equation}\label{eq:H2}
H^2(X) \cong \mathbb{Z}^6 \oplus \textrm{Im}\,\rho_0 \cong \mathbb{Z}^6 \oplus \langle e\rangle,\tag{$\diamondsuit$}
\end{equation}
and therefore $\langle e\rangle \cong \textrm{Tor}\big(H_1(X)\big) = \mathbb{Z}_m$. The part
\[ \mathbb{Z}_m = H^5(X) \to H^5(M\times S^1) \to H^4(X)\to \textrm{Im}\,\rho_4 = 0 \]
implies that $\textrm{Im}\big[H^5(X) \to H^5(M\times S^1)\big] = 0$ and $H^4(X)\cong H^5(M \times S^1) \cong \mathbb{Z}^6$. Again by Poincar\'e duality, $H_2(X) \cong H^4(X) \cong \mathbb{Z}^6$ and $H^3(X) \cong H_3(X)$, and $H_3(X)$ is a free abelian group by the universal coefficients theorem. However, the part
\[ 0 = \textrm{Im}\,\rho_2 \to H^4(X) \to H^4(M \times S^1) \to H^3(X) \to \textrm{Im}\,\rho_3 \to 0 \]
tells us that $H^3(X)$ is a torsion group. Consequently, $H^3(X)=0$. Finally, the part
\[ 0 = H^3(X) \to H^3(M\times S^1) \to H^2(X) \to \textrm{Im}\,\rho_2 =0 \]
shows that $H^2(X) \cong \mathbb{Z}^6$. Comparing this with \eqref{eq:H2}, we conclude that $c_1(\xi) = e = 0$.\medskip

\noindent\emph{Proof of Claim 2.} Thanks to equivalence of bundles, we already have a diffeomorphism $M\times S^1 \to X\times S^1$. Lift it to a map of universal covers $\varphi\colon M \times \mathbb{R} \to X \times \mathbb{R}$. Since $\varphi\big(M \times \{0\}\big)$ is compact, it is contained in $X \times (-a, a)$ for some $a>0$. Let
\[ W=\varphi\big(M\times[0,\infty)\big)\cap \big(X\times (-\infty, a]\big).\] 
Then $\partial W = \big(X\times \{a\}\big)\sqcup \varphi\big(M\times\{0\}\big)$, hence $W$ is a cobordism between $X$ and~$M$. The inclusions $X\hookrightarrow W$ and $\varphi(M)\hookrightarrow W$ are homotopy equivalences. Given that $\pi_1(M)=0$, the $h$-cobordism theorem yields a diffeomorphism $M\to X$.
\end{proof}

\begin{remark}
We want to thank the referee for supplying us with the above version of the proof of Claim 1 in Proposition \ref{prop:product-actions} and pointing out that our earlier approach also did not cover the case $\pi_1\big((M\times S^1)/S^1\big) \cong \mathbb{Z}$. However, we still believe that it is true that \textit{every} free circle on $M\times S^1$ is a product one.
\end{remark}

\subsection{Detecting non-product free actions on $M \times S^1$}

Recall that if $G$ is a finite group which acts freely on a manifold $X$, then the orbit space $X/G$ can be endowed with a unique smooth structure such that the projection $X \to X/G$ is a smooth covering. Whenever we consider such an orbit space, it is this smooth structure that we have in mind.

The following theorem provides a criterion for deciding whether a free orientation preserving $\mathbb{Z}_p$-action on $M \times S^1$ is a product one or not:

\begin{theorem}\label{thm:detecting_non-product_actions}
A free orientation preserving $\mathbb{Z}_p$-action on $M\times S^1$ is a product action if and only if its orbit space is diffeomorphic to $M\times S^1$.
\end{theorem}

\noindent We will now work towards showing that the fundamental group of the orbit space of a free $\mathbb{Z}_p$-action on $M \times S^1$ is infinite cyclic. Once we have that, Theorem \ref{thm:detecting_non-product_actions} will follow from an argument similar to the one of Jahren and Kwasik \cite[\mbox{Section 4}]{Jahren2010}. (See also Khan \cite{Khan2016}.)

\begin{lemma}\label{universal_cover}
Let $p$ be a prime. There is no free $\mathbb{Z}_{p}$-action on $M \times \mathbb{R}$.
\end{lemma}

\begin{proof}
Assume, on the contrary, that such an action exists. If it induces the identity on $H^6(M \times \mathbb{R}; \mathbb{Z})$, then it is cohomologically trivial by a result of Puppe \cite[Theorem 4]{Puppe2007}, which states that the ring $H^*(M;\mathbb{Z})$ does not admit any non-trivial orientation preserving automorphisms of finite order. Consequently, the local coefficient system in the Leray--Serre spectral sequence 
\[ E_2^{s,t} \cong H^s\big(B\zmod{p}; H^t(M\times\mathbb{R};\mathbb{Z})\big) \Longrightarrow H^{s+t}\big(E\zmod{p} \times_{\zmod{p}} (M \times \mathbb{R}); \mathbb{Z}\big)
\]
of the Borel fibration $M\times\mathbb{R} \to E\zmod{p} \times_{\zmod{p}} (M \times \mathbb{R}) \to B\zmod{p}$ is trivial. Recall that cohomology groups $H^*(M; \mathbb{Z})$ are trivial in odd dimensions, so that the groups $E_2^{s,t}$ vanish for odd $t$. Furthermore, since $H^*(M;\mathbb{Z})$ are finitely-generated free abelian in even dimensions and
\[ H^k(B\mathbb{Z}_p, \mathbb{Z})=\begin{cases}
\mathbb{Z}, & k=0,\\
0, & \text{$k$ is odd},\\
\mathbb{Z}_p, & \textnormal{$k > 0$ is even},
\end{cases}\]
it follows from the universal coefficients theorem that the groups $E_2^{s,t}$ vanish also for odd $s$. Consequently, the spectral sequence collapses at the $E_2$-page. This implies that the cohomological dimension of $E\zmod{p} \times_{\zmod{p}} (M \times \mathbb{R}) \simeq (M\times\mathbb{R})/\mathbb{Z}_p$ is infinite, which clearly is impossible.

It remains to deal with the cohomologically non-trivial case for $p=2$. The ring $H^*(M;\mathbb{Z})$ has a unique orientation reversing involution, given by the identity on $H^0$ and $H^4$, and by multiplication by $(-1)$ on $H^2$ and~$H^6$. In particular, the action induced on $H^*(M;\mathbb{Z}_2)$ is necessarily trivial. In view of Puppe \cite[Lemma 1]{Puppe2007}, $H^*(M;\mathbb{Z}_{2})$ does not admit derivations of negative degree. Since differentials on the $E_2$-page of the aforementioned spectral sequence considered with $\textrm{mod}\,2$ coefficients give rise to such derivations, all of the former have to vanish. This again implies that the cohomological dimension of $(M\times\mathbb{R})/\mathbb{Z}_p$ is infinite.
\end{proof}

\begin{corollary}\label{cor:pi1_orbit}
Suppose that a finite group $G$ acts freely on $M \times S^1$. Then \mbox{$\pi_1\big((M\times S^1)/G\big) \cong \mathbb{Z}$}.
\end{corollary}

\begin{proof} Denote by $X$ the orbit space $(M \times S^1)/G$. Since $\pi_1(M\times S^1) \cong \mathbb{Z}$, the long exact sequence of homotopy groups of the cover $M \times S^1 \to X$ gives rise to a short exact sequence
\[ 0 \to \mathbb{Z} \to \pi_1(X) \to G \to 1, \]
i.e. $\pi_1(X)$ is a virtually cyclic group. Since $X$ is universally covered by $M \times \mathbb{R}$, the group $\pi_1(X)$ acts on $M\times\mathbb{R}$ by deck transformations. In view of Lemma~\ref{universal_cover}, $\pi_1(X)$ is torsion free. But the only torsion free virtually cyclic group is the infinite cyclic group (see Scott and Wall \cite[Theorem 5.12]{Scott1979}), hence $\pi_1(X)\cong \mathbb{Z}$.
\end{proof}

As a by-product, we obtain:

\begin{corollary}
If a finite group $G$ acts freely on $M \times S^1$, then $G$ is cyclic.
\end{corollary}

\begin{proof}[Proof of Theorem \ref{thm:detecting_non-product_actions}] Clearly, conjugate actions have diffeomorphic orbit spa\-ces, and the orbit space of a diagonal action on $M \times S^1$ which is free and orientation preserving is again $M \times S^1$.

Conversely, let $\phi$ be a $p$-periodic, fixed point free and orientation preserving self-map of $M\times S^1$, i.e. $\phi$ induces a free orientation preserving $\mathbb{Z}_p$-action on $M\times S^1$. Denote by $\Delta$ the self-map of $M\times S^1$ which is trivial on $M$ and given by the rotation by $e^{2\pi i/p}$ on $S^1$. Write $\pi_\Delta$ for the orbit projection corresponding to $\Delta$.

Suppose that $f\colon (M\times S^1) /\phi \to (M \times S^1) /\Delta$ is a diffeomorphism. Since the fundamental group of $(M \times S^1)/\phi$ is infinite cyclic by Corollary~\ref{cor:pi1_orbit}, after choosing basepoints there exists a lift $F\colon M \times S^1 \to M \times S^1$ of $f\circ \pi_{\Delta}$. Then $\phi\circ F$ and $F\circ\Delta^k$, where $1 \leq k \leq  p-1$, are also lifts of $f$, all different from~$F$. But there are only $p$ different lifts in total, hence $\phi \circ F=F \circ \Delta^k$ for some $k$, and the action induced by $\phi$ is conjugate to a product one.
\end{proof}

Note that for $p=2$ the argument above in fact proves a slightly stronger statement:

\begin{corollary}
Free orientation preserving involutions on $M \times S^1$ are conjugate if and only if the respective orbit spaces are diffeomorphic.
\end{corollary}

\begin{remark}
Theorem \ref{thm:detecting_non-product_actions} remains true in homotopical and topological settings.
\end{remark}

\section{Conclusion}\label{sec:Conjectures}

Let us conclude with a few remarks on the expected behavior of actions on product manifolds. The smallest sphere (in terms of dimension) on which a finite group acts non-freely is the linear sphere $S(V\oplus \mathbb{R})$ for some non-trivial irreducible representation $V$. The proof of Theorem \ref{thm:non-product-actions} suggests that the ``threshold dimension'' mentioned in the introduction is intimately connected with dimension of $V$: a product action on $M\times S(V\oplus\mathbb{R})$ can be readily modified to an exotic one with methods at hand. In general, as soon as we grasp the fixed point set (with its normal representation) of a locally linear $G$\nobreakdash-action, we can apply the same principle to alternate the action.

On the other hand, free $G$-actions on $M\times S^n$ give rise to coverings of orbit spaces $(M\times S^n)/G$. While all of these are of the same homotopy type, it seems unlikely that they share the diffeomorphism type. Smooth classification of orbit spaces is therefore the key to classification of actions on $M\times S^n$. However, the classification of such manifolds seems to be a challenging problem even for $M\in \mathcal{M}_\textsl{As}$. Nonetheless, we believe that in this case any free $\mathbb{Z}_{p}$\nobreakdash-action on $M\times S^1$ is conjugate to a product one.\medskip

\noindent\textbf{Acknowledgements.} We wish to thank the referee for a thorough review. Their detailed comments improved the final version of this manuscript immensely.

\bibliographystyle{amsplain}
\bibliography{Asymmetric.bib}\bigskip

\noindent\textsc{Zbigniew Błaszczyk}\\
Faculty of Mathematics and Computer Science\\
Adam Mickiewicz University\\
Umultowska 87\\
61-614 Poznań, Poland\\
\texttt{blaszczyk@amu.edu.pl}\medskip

\noindent\textsc{Marek Kaluba}\\
Faculty of Mathematics and Computer Science\\
Adam Mickiewicz University\\
Umultowska 87\\
61-614 Poznań, Poland\\
\texttt{kalmar@amu.edu.pl}

\end{document}